\documentclass{article}
\usepackage{stmaryrd}
\usepackage{amsmath,amsthm,amssymb}
\usepackage[all]{xy}

\newcommand{\N}{\mathbb{N}}
\newcommand{\M}{\mathrm{M}}
\newcommand{\GL}{\mathrm{GL}}
\newcommand{\F}{\mathbb{F}}
\newcommand{\magma}{Magma}

\newcommand{\andsoon}{\rule{0.6cm}{0.3pt}\ \!\raisebox{-0.2ex}{\tiny{//}}\rule{0.6cm}{0.3pt}}

\newtheorem{proposition}{Proposition}
\newtheorem{corollary}{Corollary}
\newtheorem{theorem}{Theorem}

\author{Ivan Yudin
\thanks{The work is supported by the FCT Grant SFRH/BPD/31788/2006. The
financial support by CMUC and FCT gratefully acknowledged.}
}
\title{Presentation for parabolic subgroups of $\GL_n(\F_2)$. } 
\begin{document}
\maketitle
\section{Introduction}

Let $G$ be a finite group and $X=\{x_1, x_2
,\dots, x_n\}$ be the set of generators of $G$. The set $X$ is called \emph{a local generating system of depth $k$}, if 
$$
x_ix_j = x_jx_i,\ |i-j|\ge 2.
$$
Local generating systems play important role in representation theory (see \cite{okunkov}) and in classification of finite simple groups (see \cite{phan}). 

A big class of locally presented groups is provided by Coxeter systems (see \cite{tits}). 
We define two-dimensional Coxeter systems as generalizations of the usual
Coxeter systems. For a set $X$ we denote by ${X \choose k}$ the set of
$k$-subsets of $X$. A \emph{two-dimensional Coxeter system} is a triple
$(X,f,g)$, where $f\colon {X \choose 2} \to \N$ and $g\colon {X\choose 3} \to N$
are arbitrary functions. We will consider $f$ as a function on $X^2$ invariant
under swapping the arguments and $g$ as a function on $X^3$ invariant under the
action of $\Sigma_3$. To every two-dimensional Coxeter system $(X,f,g)$ we
associate a finitely presented group $\left\langle X,f,g\right \rangle$ with the
set of generators $X$ and defining relations 
$$
x^2=e,\ (xy)^{f(x,y)} = e, \ (xyz)^{g(x,y,z)} = e,\qquad x,y,z \in X.
$$
Observe that the relations $x_1^2=x_2^2=x_3^2=3$ and $(x_1x_2x_3)^k=e$ imply that
$(x_{\sigma(1)}x_{\sigma(2)}x_{\sigma(3)})^k=e$ for any $\sigma\in \Sigma_3$,
and therefore there is no ambiguity in the description of $\left\langle
X,f,g\right\rangle$. 

Recall that a stabilizer of any flag $0=V_0\subset V_1\subset V_2\subset
\dots\subset V_k=\F_2^n $ in $\GL_n(\F_2)$ is called a \emph{parabolic subgroup}
of $\GL_n(\F_2)$. The  main result of the paper is that every parabolic
subgroup of $\GL_n(F_2)$ has a presentation described by a two-dimensional Coxeter
system, moreover, this presentation is local of depth $3$. In fact, we shall
prove this result for a large class of subgroups of $\GL_n(\F_2)$, which includes
among others parabolic subgroups. It should be noted that in the case of the trivial
flag $0\subset \F_2^n$ we will recover the well-known Steinberg-Curtis presentation of
$\GL_n(\F_2)$. In the case of a maximal flag the corresponding parabolic group is
isomorphic to the group of unipotent matrices over $\F_2$ and our presentation
coincides with the one obtained in~\cite{biss}.

During the preparation of the paper the computations were performed with
computer algebra system \magma\cite{magma}. 
\section{Two-dimensional Coxeter systems}
Every Coxeter system $(X,S)$ corresponds to a complete graph with integer labels on
edges, which is called a Coxeter diagram. The vertices of the Coxeter diagram are the
elements of $X$ and the label of the edge $(xy)$ is $f(x,y)$. Similarly, we
associate to every two-dimensional Coxeter system $(X,f,g)$ a complete
two-dimensional simplicial complex $C(X,f,g)$ with labeled edges and facets. We
shall refer to $C(X,f,g)$ as \emph{a Coxeter diagram} of $(X,f,g)$. The vertices
of $C(X,f,g)$ are the elements of $X$, the edge $(xy)$ will have the label $f(x,y)$ and
the facet $(xyz)$ will have the label $g(x,y,z)$.

To
make these diagrams human comprehensible we shall use the following conventions:
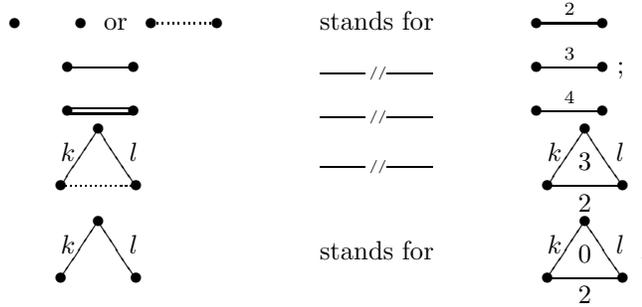
\begin{tabbing}
$\xymatrix{\bullet \!\!\! & \!\!\! \bullet}$ or $\xymatrix{\bullet\!\!\! \ar@{.}[r] &
\!\!\!\bullet}$ \quad\quad\quad \= stands for \quad\quad\quad \= 
$
\xymatrix{ \bullet\!\!\! \ar@{-}[r]^2 & \!\!\!\bullet}
$
\\
\quad \quad $\xymatrix{\bullet \!\!\!  \ar@{-}[r]&\!\!\!   \bullet}$ \>\andsoon 
 \>$
\xymatrix{ \bullet\!\!\! \ar@{-}[r]^3 & \!\!\!\bullet};
$
\\
\quad\quad$\xymatrix{\bullet \!\!\!  \ar@{=}[r]&\!\!\!  \bullet}$ \>
\andsoon
 \>$
\xymatrix{ \bullet\!\!\! \ar@{-}[r]^4 & \!\!\!\bullet}
$
\\
\quad\quad$
\raisebox{-0.25cm}{\xy<1cm,0cm>:
(0,0)*=0{\bullet} = "0";
(1,0)*=0{\bullet} = "1" **@{.};
(0.5,0.756)*=0{\bullet} = "3";"1" **@{-}?(0.4)*++!L{l};
"0";"3" **@{-}?(0.6)*+!R{k}
\endxy}\ \ 
$
 \>\andsoon
 \>$
\ \ 
\raisebox{-0.25cm}{\xy<1cm,0cm>:
(0,0)*=0{\bullet} = "0";
(1,0)*=0{\bullet} = "1" **@{-}?*+!U{2};
(0.5,0.756)*=0{\bullet} = "3";"1" **@{-}?(0.4)*++!L{l};
"0";"3" **@{-}?(0.6)*+!R{k};
(0.5,0.328)*=0{3}
\endxy}
$
\\
\quad\quad$
\raisebox{-0.25cm}{\xy<1cm,0cm>:
(0,0)*=0{\bullet} = "0";
(1,0)*=0{\bullet} = "1" **@{};
(0.5,0.756)*=0{\bullet} = "3";"1" **@{-}?(0.4)*++!L{l};
"0";"3" **@{-}?(0.6)*+!R{k}
\endxy}\ \ 
$
 \>stands for
 \>$
\ \ 
\raisebox{-0.25cm}{\xy<1cm,0cm>:
(0,0)*=0{\bullet} = "0";
(1,0)*=0{\bullet} = "1" **@{-}?*+!U{2};
(0.5,0.756)*=0{\bullet} = "3";"1" **@{-}?(0.4)*++!L{l};
"0";"3" **@{-}?(0.6)*+!R{k};
(0.5,0.328)*=0{0}
\endxy}
$.
\end{tabbing}

For every $n$, we define a two-dimensional Coxeter system $A_{2,n}$ with the set
$X=\left\{ x_1,\dots,x_{n},y_1,\dots,y_n \right\}$. The function $f$ is defined as follows
$$
f(x_i,x_{i+1}) = f(y_i,y_{i+1}) = 4,\ 1\le i \le n-1;
$$
$$
f(x_i,y_i) =3,\ 1 \le i \le n,
$$
and $f(z,t)$ is $2$ for all other pairs of elements in $X$. The function
$g$ is given by
$$
g(x_i,x_{i+1},y_i) = g(x_i,x_{i+1},y_{i+1}) = g(x_i,y_i,y_{i+1}) =
g(x_{i+1},y_i,y_{i+1}) =3, \ 1\le i \le n-1;
$$
$$
g(x_i,x_{i+1},x_{i+2}) = g(y_i,y_{i+1},y_{i+2}) =4,\ 1\le i\le n-2
$$
and $g(z,t,v)=0$ for any other triple of elements in $X$.  
The corresponding diagrams for small values of $n$ have the form:
$$
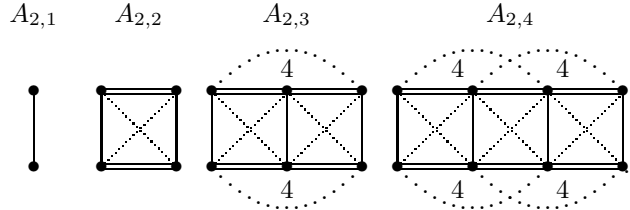
\xy<2cm,0cm>:
(0,1)*+{A_{2,1}};
(0,0.5)*=0{\bullet};
(0,0)*=0{\bullet} **@{-}
\endxy
\ \ \ \ 
\xy<2cm,0cm>:
(0.25,1)*+{A_{2,2}};
(0,0.5)*=0{\bullet}="u1";
(0,0)*=0{\bullet}="l1";
(0.5,0.5)*=0{\bullet}="u2";
(0.5,0)*=0{\bullet}="l2";
"u1";"l1"**@{-};
"u2";"l2"**@{-};
"u1";"u2"**@{=};
"l1";"l2"**@{=};
"u1";"l2"**@{.};
"u2";"l1"**@{.}
\endxy
\ \ \ \ 
\xy<2cm,0cm>:
(0.5,1)*+{A_{2,3}};
(0,0.5)*=0{\bullet}="u1";
(0,0)*=0{\bullet}="l1";
(0.5,0.5)*=0{\bullet}="u2";
(0.5,0)*=0{\bullet}="l2";
(1,0.5)*=0{\bullet}="u3";
(1,0)*=0{\bullet}="l3";
"u1";"l1"**@{-};
"u3";"l3"**@{-};
"u2";"l2"**@{-};
"l1";"l2"**@{=};
"u1";"u2"**@{=};
"u1";"l2"**@{.};
"u2";"u3"**@{=};
"l2";"l3"**@{=};
"u2";"l1"**@{.};
"u2";"l3"**@{.};
"u3";"l2"**@{.};
"u1";"u3"**\crv{~*=<5pt>{.} (0.5,1.05)};
"l1";"l3"**\crv{~*=<5pt>{.} (0.5,-0.55)};
(0.5,0.65)*+{4};
(0.5,-0.15)*+{4}
\endxy
\ \ \ \ 
\xy<2cm,0cm>:
(0.75,1)*+{A_{2,4}};
(0,0.5)*=0{\bullet}="u1";
(0,0)*=0{\bullet}="l1";
(0.5,0.5)*=0{\bullet}="u2";
(0.5,0)*=0{\bullet}="l2";
(1,0.5)*=0{\bullet}="u3";
(1,0)*=0{\bullet}="l3";
(1.5,0.5)*=0{\bullet}="u4";
(1.5,0)*=0{\bullet}="l4";
"u1";"l1"**@{-};
"u4";"l4"**@{-};
"u3";"l3"**@{-};
"u2";"l2"**@{-};
"l1";"l2"**@{=};
"u1";"u2"**@{=};
"u3";"u4"**@{=};
"l3";"l4"**@{=};
"u1";"l2"**@{.};
"u3";"l4"**@{.};
"u4";"l3"**@{.};
"u2";"u3"**@{=};
"l2";"l3"**@{=};
"u2";"l1"**@{.};
"u2";"l3"**@{.};
"u3";"l2"**@{.};
"u1";"u3"**\crv{~*=<5pt>{.} (0.5,1.05)};
"l1";"l3"**\crv{~*=<5pt>{.} (0.5,-0.55)};
(0.4,0.65)*+{4};
(1.1,0.65)*+{4};
(0.4,-0.15)*+{4};
(1.1,-0.15)*+{4};
"u2";"u4"**\crv{~*=<5pt>{.} (1,1.05)};
"l2";"l4"**\crv{~*=<5pt>{.} (1,-0.55)};
\endxy.
$$
For a subset $S$ of $X$, we denote by $A_{2,n}(S)$ the two-dimensional Coxeter
system obtained from $A_{2,n}$ by restriction of $f$ and $g$ on $S$. We will
identify the groups  $A_{2,n}(S)$ with certain subgroups of $\GL_{n+1}(\F_2)$.  

\section{Parabolic subgroups and their intersections}
Let $\lambda$ be a decomposition of $n$. We denote by $P_{\lambda}$ the standard
parabolic subgroup of $\GL_n(\F_2)$ corresponding to $\lambda$. The group
$P_\lambda$ has the form:
$$
\left( 
\begin{array}{cccc}
	GL_{\lambda_1}(\F_2) & \M_{\lambda_1,\lambda_2}(\F_2) & \dots &
	\M_{\lambda_1,\lambda_l}(\F_2) \\
	0 & \GL_{\lambda_2}(\F_2) & \dots & \M_{\lambda_2,\lambda_l}(\F_2) \\
	\vdots & \vdots & \ddots & \vdots \\
	0 & 0 & \dots & \GL_{\lambda_l}(\F_2)
\end{array}
\right),
$$
where $l$ is the length of $\lambda$. We denote by $P_{\lambda|\mu}$ the
intersection $P_\lambda\cap P_\mu^t$. Observe that $P_{\lambda|\lambda}$ is the
standard
Levi subgroup of $\GL_n(\F_2)$ corresponding to $\lambda$. 
Note that these subgroups (over an arbitrary field) have appeared in the
PhD thesis~\cite{woodcock} of D.~Woodcock. The groups $P_{\lambda|\mu}$ are the
main subject of this paper. 

Let us describe recursive formulas for the orders of the groups $P_{\lambda|\mu}$. 
We shall always assume that $\lambda$ has $l$ non-zero parts, and $\mu$ has
$m$ non-zero parts. 
\begin{proposition}
	\label{order}
	Suppose that $\mu_m\le \lambda_l$. Then
	$$
	\left| P_{\lambda|\mu} \right| = 
	 2^{\mu_m(\lambda_l-\mu_m)} \left|
	\GL_{\mu_m}(\F_2)\right|\cdot \left| P_{\widetilde{\lambda}|\widetilde{\mu}}
	\right|,
	$$
	where $\widetilde{\lambda} =
	(\lambda_1,\lambda_2,\dots,\lambda_l-\mu_m)$ and $\widetilde{\mu} =
	(\mu_1,\mu_2,\dots,\mu_{m-1})$. 
\end{proposition}
\begin{proof}
	Every element of $P_{\lambda|\mu}$ can be written as a block matrix
	$$
	\left(\begin{array}{cc} A & 0 \\  C & B \end{array}\right),
	$$
	where $A\in P_{\widetilde{\lambda}| \widetilde{\mu}}$, $B \in
	\GL_{\mu_m}(\F_2)$ and $C\in \M_{\mu_m,n-\mu_m}(\F_2)$. Moreover, only
	the last $\lambda_l-\mu_m$ columns of $C$ are non-zero. On the other
	hand, for any given $A\in P_{\widetilde{\lambda}| \widetilde{\mu}}$, $B \in
	\GL_{\mu_m}(\F_2)$, $C\in \M_{\mu_m,n-\mu_m}(\F_2)$ such that only the
	last $\lambda_l-\mu_m$ columns of $C$ are non-zero, the matrix
	$$
	\left(\begin{array}{cc} A & 0 \\  C & B \end{array}\right)
	$$
	is invertible and thus an element of $P_{\lambda|\mu}$. Therefore there
	is a one-to-one correspondence between  the elements of $P_{\lambda|\mu}$ and
	 the elements of the Cartesian product 
	$$
	P_{\widetilde{\lambda}|\widetilde{\mu}} \times \GL_{\mu_m}(\F_2) \times
	\M_{\mu_m, \lambda_l-\mu_m} (\F_2).
	$$
\end{proof}
Suppose $\mu_m\ge 2$. We shall denote by $\mu'$ the decomposition
$$(\mu_1,\dots,\mu_{m-1}, \mu_m-1,1)$$ of $n$. 

\begin{corollary}
	\label{order1}
	Suppose $1< \mu_m< \lambda_l$. Then $\left[ P_{\lambda|\mu} :
	P_{\lambda|\mu'} \right] = 2^{\mu_m} -1$. 
\end{corollary}
\begin{proof}
	Applying Proposition~\ref{order} two times we get
	\begin{align*}
		\left|P_{\lambda|\mu'}\right| & = 2^{1\cdot(\lambda_l-1)}
		\left|\GL_1(\F_2)\right|\cdot \left|
		P_{(\lambda_1,\dots,\lambda_l-1)|(\mu_1,\dots,\mu_m-1)}\right|\\
		&=
		2^{(\lambda_l-1)}
		2^{(\mu_m-1)(\lambda_l-\mu_m)}\left|\GL_{\mu_m-1}(\F_2)\right|\cdot
		\left|P_{\widetilde{\lambda}|\widetilde{\mu}}
	\right|.
\end{align*}
Applying Proposition~\ref{order} one more time and using well-known formulas
for the orders of general linear groups we obtain
\begin{align*}
	\frac{\left|P_{\lambda|\mu}\right| }{\left|P_{\lambda|\mu'}\right|} & =
	\frac{2^{\mu_m(\lambda_l-\mu_m)} \left| \GL_{\mu_m}(\F_2)\right|
	}{2^{(\mu_m-1)(\lambda_l-\mu_m)} 2^{\lambda_l-1}\left|
	\GL_{\mu_m}(\F_2)\right|} \\
	&= \frac{2^{\lambda_l-\mu_m}2^{\mu_m-1} (2^{\mu_m} -1)}{2^{\lambda_l-1}}
	= 2^{\mu_m} -1.
\end{align*}
\end{proof}
\begin{corollary}
	\label{order2}
	Suppose $\lambda_l=1$. Then the index of $P_{\lambda|\mu'}$ in $P_{\lambda,\mu}$ is
	$2^{\mu_m-1}$.
\end{corollary}
\begin{proof}
	Since $\lambda_l=1\le \mu_m$ we can apply Proposition~\ref{order} with
	the roles of $\lambda$ and $\mu$ swapped.  We get
	\begin{align*}
		\left|P_{\lambda|\mu}\right| & = 2^{1\cdot (\mu_m-1)}
		\left|\GL_1(\F_2)\right|\cdot 
		\left|P_{(\lambda_1,\cdots,\lambda_{l-1})|(\mu_1,\cdots,\mu_m-1)}\right|
	\\&= 
	2^{\mu_m-1}\left|P_{\lambda|\mu'} \right|.
	\end{align*}
		\end{proof}
\section{Generators and relations}
For every decomposition $\lambda=(\lambda_1,\dots,\lambda_l)$ of $n+1$ we
define the subset of stopovers
$$
\overline{s}(\lambda) = \left\{
\lambda_1,\lambda_1+\lambda_2,\dots,\lambda_1+\lambda_2+\dots+\lambda_{l-1}
\right\}.
$$
As it is always assumed that $\lambda_l\not = 0$ we have $\overline{s}(\lambda)\subset
\{1,\dots,n-1\}$. Now for every pair $(\lambda,\mu)$ of decompositions of
$n+1$ we define a  subset $S_{\lambda|\mu}$ of
$X=\left\{x_1,\dots,x_n,y_1,\dots,y_n  \right\}$ by
$$
S_{\lambda|\mu} = \left\{ x_i \middle| i \not \in \overline{s}(\lambda) \right\}
\cup \left\{ y_j \middle| j \not\in \overline{s}(\mu) \right\}.
$$

\begin{theorem}
	\label{main}
	Let $\lambda$ and $\mu$ be decompositions of $n+1$. The groups
	$P_{\lambda|\mu}$ and $\left\langle A_{2,n}(S_{\lambda|\mu})\right\rangle$ are
	isomorhic.
\end{theorem}
The rest of this section is devoted to proving Theorem~\ref{main}. The strategy
of the proof is the usual one: for every $\lambda$ and $\mu$ we define a homomorphism
$$\phi_{\lambda,\mu}\colon
\left\langle A_{2,n}(S_{\lambda|\mu}) \right\rangle\to P_{\lambda|\mu},$$ then we
show that $\phi_{\lambda,\mu}$ is surjective and deduce from  the
comparison of orders that $\phi_{\lambda,\mu}$ are isomorphisms. 

Define $\phi\colon \left\{ x_j, y_j \ \middle|\  1\le j\le n \right\}\to
\GL_{n+1}(\F_2)$ by
\begin{align*}
	\phi(x_j) & := I_{j-1}\oplus \left( 
	\begin{array}{cc}
		1 & 0\\ 
		1 & 1
	\end{array}
	\right)\oplus I_{n-j} = I_{n+1} + E_{j+1,j}\\
	\phi(y_j) & := I_{j-1}\oplus \left( 
	\begin{array}{cc}
		1 & 1 \\
		0 & 1
	\end{array}
	\right)\oplus I_{n-j} = I_{n+1} + E_{j,j+1},
\end{align*}
where $E_{i,j}$ denotes the elementary matrix with $1$ at $i$-th row and
$j$-th column.
\begin{proposition}
The map $\phi$ can be extended to a homomorphism of groups $\phi\colon
\left\langle A_{2,n} \right\rangle\mapsto \GL_n(\F_2)$.	
\end{proposition}
\begin{proof}
We have to check that all relations imposed by the two-dimensional Coxeter
system $A_{2,n}$ on the generators of $\left\langle A_{2,n} \right\rangle$ hold
for the elements $\phi(x_j)$, $\phi(y_j)$ of $\GL_n(\F_2)$. First, we observe that
if $|j-k|\ge 2$ then $\phi(x_j)$ and $\phi(y_j)$ commute with
$\phi(x_k)$ and $\phi(y_k)$, as required. Now all other non-trivial relations
involve only  the elements $x_j$, $x_{j+1}$, $x_{j+2}$, $y_j$, $y_{j+1}$, $y_{j+2}$  for some
$j$. It is clear from the structure of the matrices $\phi(x_j)$, $\phi(y_j)$ 
that it is enough to check validity of the relations only for the matrices
$\phi(x_1)$, $\phi(x_2)$, $\phi(x_3)$, $\phi(y_1)$, $\phi(y_2)$, $\phi(y_3)$.
This can be done either by direct hand computation or with the help of any
computer system, for example, \magma. 
\end{proof}
Before we proceed, let us remark that if a subgroup $H$ of
$\GL_{n+1}(\F_2)$ contains $I_{n+1}+ E_{i,j}$ and $I_{n+1} + E_{j,k}$ with
$i<j<k$ (or $i>j>k$), then $I+E_{i,k} \in H$. In fact,
\begin{align*}\left( \left( I+E_{i,j} \right)\left( I+E{j,k} \right) \right)^2 & = 
	\left( I+E_{i,j}+E_{j,k} + E_{i,k} \right)^2\\ & =
	I + E_{i,j}+E_{j,k}+E_{i,k} + E_{i,j} + E_{i,k} + E_{j,k} + E_{i,k}\\ & = I
+
E_{i,k}.
\end{align*}
Therefore if $\left\{ x_i, x_{i+1},\dots, x_j \right\} \subset S_{\lambda|\mu}$
then $I_{n+1}+E_{j+1,i}\in \left\langle A_{2,n}\left( S_{\lambda|\mu} \right)
\right\rangle $ and if $\left\{ y_i, y_{i+1},\dots, y_j \right\}\subset
S_{\lambda|\mu}$ then $I_{n+1} + E_{i,j+1} \in \left\langle A_{2,n}\left(
S_{\lambda|\mu}
\right) \right\rangle$.
We define $\phi_{\lambda,\mu}$ to be the restriction of $\phi$ on $\left\langle
A_{2,n}(S_{\lambda|\mu}) \right\rangle$. 
\begin{proposition}
	The image of $\phi_{\lambda,\mu}$ coincides with $P_{\lambda|\mu}$. 
\end{proposition}
\begin{proof}
	First we show that $\phi(S_{\lambda|\mu})$ is a subset of
	$P_{\lambda|\mu}$. This will imply that the image of
	$\phi_{\lambda,\mu}$ is a subgroup of $P_{\lambda|\mu}$. Recall that
	$P_{\lambda|\mu} = P_\lambda\cap P_\mu^t$. Since $P_{\mu}^t$ contains
	all lower-triangular matrices and $\phi(x_j)$ are lower-triangular it follows that $\phi(x_j)\in P_{\mu}^t$
	for all $j$. Taking into account the symmetry between $x$'s and
	$y$'s and between $\lambda$ and $\mu$ it is enough to show that
	$\phi(x_j)$ is an element of $P_\lambda$
	for all $x_j \in S_{\lambda|\mu}$. From the picture
	$$
	
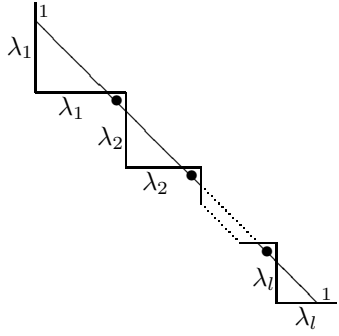
\xy<4cm,0cm>:
	(0,1);(0,0.7) **@{-};
	(0.3,0.7) **@{-};
	(0.3,0.45) **@{-};
	(0.55,0.45) **@{-};
	(0.55,0.33) **@{-};
	(0.68,0.2) **@{.};
	(0.8,0.2) **@{-};
	(0.8,0) **@{-};
	(1,0) **@{-};
	(0,0.94);(0.55,0.39) **@{-}; 
	(0.74,0.2) **@{.};
	(0.94,0) **@{-};
	(0.27,0.67)*=0{\bullet};
	(0.52,0.42)*=0{\bullet};
	(0.77,0.17)*=0{\bullet};
	(-0.05,0.85)*+{\lambda_1};
	(0.03,0.97)*+{ {}_{1}};
	(0.12,0.65)*+{\lambda_1};
	(0.25,0.55)*+{\lambda_2};
	(0.4,0.4)*+{\lambda_2};
	(0.76,0.08)*+{\lambda_l};
	(0.9,-0.05)*+{\lambda_l};
	(0.97,0.03)*+{ {}_1}; 
	\endxy
	$$
	it follows immediately  that $\phi(x_j) \in P_{\lambda}$ if and only if
	$j \not\in \overline{s}(\lambda)$, which is equivalent to $x_j\in
	S_{\lambda|\mu}$. 

	To prove that $\phi_{\lambda,\mu}$ is surjective on $P_{\lambda|\mu}$  we proceed by induction on
	$n$. For $n=0$ there is nothing to prove. Suppose $n\ge 1$. Without loss
	of generality we can assume that $\mu_m\le \lambda_l$. Then every
	element of $P_{\lambda|\mu}$ can be written as a block matrix
	$$
	\left( 
	\begin{array}{cc}
		A & 0 \\
		C & B
	\end{array}
	\right),
	$$
	where $A\in P_{\widetilde{\lambda}|\widetilde{\mu}}$, $B\in
	\GL_{\mu_m}(\F_2)$, $C\in \M_{\mu_m,n+1-\mu_m}$, and $C$ is such that only the last
	$\lambda_l-\mu_m$ columns of $C$ are non-zero. By the remark above the
	matrix $I_{n+1-\mu_m}\oplus B$ can be written as a product of
	matrices $\phi(x_j)$ and $\phi(y_j)$ with $j\ge \mu_1 + \dots
	+\mu_{m-1}+1 = n-\mu_m+2$, and therefore $I_{n+1-\mu_m}\oplus B$ lies in the
	image of $\phi_{\lambda,\mu}$. 
	Thus it is enough to show that 
	for $A\in P_{\widetilde{\lambda}|\widetilde{\mu}}$ 
	and $C\in \M_{\mu_m,n+1-\mu_m}$ the matrix
	$$
		\left( 
	\begin{array}{cc}
		A & 0 \\
		C & I_{\mu_m}
	\end{array}
	\right)
	$$
	lies in the image of $\phi_{\lambda,\mu}$ if only the first
	$n+1-\lambda_l$
	columns of $C$ are zero.
	Multiplying the
	matrix above by $\phi(x_j)$, $j\ge \lambda_1+\dots+\lambda_{l-1}+1$, from
	the right hand side,
	we can annihilate all elements of $C$. Therefore, it remains to show
	that
$$
		\left( 
	\begin{array}{cc}
		A & 0 \\
		0 & I_{\mu_m}
	\end{array}
	\right)
	$$
	lies in the image of $\phi_{\lambda,\mu}$. This follows from the
	induction hypothesis. 
\end{proof} 
To show that the order of the group $\left\langle A_{2,n}(S_{\lambda|\mu})
\right\rangle$ does not exceed the order of $P_{\lambda|\mu}$ we proceed as
follows. We know from Corollary~\ref{order1} and
Corollary~\ref{order2} that 
\begin{align*}
	\left| P_{\lambda|\mu}\right|& = (2^{\mu_m} -1) \left|
	P_{\lambda|\mu'}\right|, && \mbox{if $1<\mu_m\le \lambda_l$,}\\
	\left| P_{\lambda|\mu} \right| & = 2^{\lambda_m-1} \left|
	P_{\lambda|\mu'}\right|, && \mbox{if $\lambda_l=1$},
\end{align*}
where $\mu'=(\mu_1,\dots,\mu_m-1,1)$.
Moreover, it is clear that 
$$
\left|P_{\lambda|\mu}\right| = \left|
P_{\widetilde{\lambda}|\widetilde{\mu}} \right|,\ \mbox{for $\mu_m=\lambda_l=1$},
$$
where $\widetilde{\lambda}=(\lambda_1,\dots,\lambda_{l-1})$ and
$\widetilde{\mu}=(\mu_1,\dots,\mu_{m-1})$. Taking into account that
$\left|P_{\lambda|\mu}\right|= \left| P_{\mu|\lambda}\right|$ these relations
provide enough information to compute the order of $P_{\lambda|\mu}$ for any
given pair of decompositions $\lambda$ and $\mu$. 

Now it is clear that the groups $\left\langle A_{2,n}(S_{\lambda|\mu})
\right\rangle$ and $\left\langle A_{2,n}(S_{\mu|\lambda}) \right\rangle$ are
isomorphic, and that for $\lambda_l=\mu_m=1$ the groups $\left\langle
A_{2,n}(S_{\lambda|\mu})
\right\rangle$ and $\left\langle
A_{2,n}(S_{\widetilde{\lambda}|\widetilde{\mu}}) \right\rangle$ are isomorphic
as well. To get the required inequality for the orders it now suffices to show
that
\begin{align*}
	\left|\left\langle A_{2,n}(S_{\lambda|\mu}) \right\rangle\right| & \le
	(2^{\mu_m} -1)\left|\left\langle A_{2,n}(S_{\lambda|\mu'})
	\right\rangle\right|
	&& \mbox{if $1<\mu_m\le \lambda_l$,}\\
	\left|\left\langle A_{2,n}(S_{\lambda|\mu}) \right\rangle\right| & \le 
	2^{\lambda_m-1} \left|\left\langle A_{2,n}(S_{\lambda|\mu'})
	\right\rangle \right|  && \mbox{if $\lambda_l=1$}.
\end{align*}
 To prove these relations we shall give a 
description of coset representatives of $\left\langle A_{2,n}(S_{\lambda|\mu'})
\right\rangle$ in $\left\langle A_{2,n}(S_{\lambda|\mu}) \right\rangle$ for
$1<\mu_m\le \lambda_l$ and of $\left\langle A_{2,n}(S_{\lambda,\mu'})
\right\rangle$ in $\left\langle A_{2,n}(S_{\lambda|\mu}) \right\rangle$ for
$\lambda_l=1$.
\begin{proposition}
	Suppose $1<\mu_m\le \lambda_l$. Denote the group $\left\langle
	A_{2,n}(S_{\lambda|\mu'})
	\right\rangle$ by $H$. Then the cosets of $H$ in $\left\langle
	A_{2,n}(S_{\lambda|\mu})	\right\rangle$ are given by the set
	$$
	\left\{ H \right\}\cup \left\{ Hy_nx_n^{\varepsilon_n}\dots
	y_kx_k^{\varepsilon_k}\  \middle|
	\ n-\mu_m+2\le k\le n,\ \varepsilon_j\in \left\{ 0,1 \right\}\right\}.
	$$
	Note that the number of elements in this set is less then or equal to
	$$
	1 + (2 + \dots + 2^{\mu_m-1}) = 2^{\mu_m}-1.
	$$
\end{proposition}
\begin{proof}
	It suffices to show that the given set of cosets is closed under the
	action of generators $z\in S(\lambda,\mu)$.
	Suppose $j\le k-2$, then 
	\begin{align*}
		Hy_nx_n^{\varepsilon_n}\dots
	y_kx_k^{\varepsilon_k} x_j &= Hy_nx_n^{\varepsilon_n}\dots
	y_kx_k^{\varepsilon_k},\\
Hy_nx_n^{\varepsilon_n}\dots
	y_kx_k^{\varepsilon_k} y_j &= Hy_nx_n^{\varepsilon_n}\dots
	y_kx_k^{\varepsilon_k},
	\end{align*}
since $x_j$, $y_j$ commute with $x_s$ and $y_s$ as soon as $|j-s|\ge 2$. 
If $y_{k-1}\in S(\lambda|\mu)$ then $Hy_nx_n^{\varepsilon_n}\dots
y_kx_k^{\varepsilon_k}y_{k-1}$ is the coset of the allowed form.

Suppose $k\le n$ and 
$k\le j\le n$. Then 
if $\varepsilon_j=0$ we have
\begin{align*}
Hy_nx_n^{\varepsilon_n}\dots
	y_kx_k^{\varepsilon_k}\cdot y_j &= Hy_nx_n^{\varepsilon_n}\dots
	y_jx_j^{0}y_{j-1}y_j
	x_{j-1}^{\varepsilon_{j-1}}y_{j-2}x^{\varepsilon_{j-2}}\dots
	y_kx_k^{\varepsilon_k}\\
&=Hy_nx_n^{\varepsilon_n}\dots
	y_jy_{j-1}y_j
	x_{j-1}^{\varepsilon_{j-1}}y_{j-2}x^{\varepsilon_{j-2}}\dots
	y_kx_k^{\varepsilon_k}\\
	&= Hy_nx_n^{\varepsilon_n}\dots
	y_j{\bf{x_{j-1}}}y_{j-1}y_j
	x_{j-1}^{\varepsilon_{j-1}}y_{j-2}x^{\varepsilon_{j-2}}\dots
	y_kx_k^{\varepsilon_k}\\
&= Hy_nx_n^{\varepsilon_n}\dots
	{\bf y_j{{x_{j-1}}}y_{j-1}y_j}
	x_{j-1}^{\varepsilon_{j-1}}y_{j-2}x^{\varepsilon_{j-2}}\dots
	y_kx_k^{\varepsilon_k}\\
&= Hy_nx_n^{\varepsilon_n}\dots
{\bf y_{j-1}x_{j-1}y_jy_{j-1}x_{j-1}}
	x_{j-1}^{\varepsilon_{j-1}}y_{j-2}x^{\varepsilon_{j-2}}\dots
	y_kx_k^{\varepsilon_k}\\
&= Hy_nx_n^{\varepsilon_n}\dots
y_jy_{j-1}
	x_{j-1}^{\varepsilon_{j-1}+1}y_{j-2}x^{\varepsilon_{j-2}}\dots
	y_kx_k^{\varepsilon_k}.
\end{align*}
If $\varepsilon_j=1$ then we get
\begin{align*}
	Hy_nx_n^{\varepsilon_n}\dots
	y_kx_k^{\varepsilon_k}\cdot y_j &= Hy_nx_n^{\varepsilon_n}\dots
	y_jx_jy_{j-1}y_j
	x_{j-1}^{\varepsilon_{j-1}}y_{j-2}x^{\varepsilon_{j-2}}\dots
	y_kx_k^{\varepsilon_k}\\
	&= Hy_nx_n^{\varepsilon_n}\dots
	{\bf y_j x_jy_{j-1}y_j}
	x_{j-1}^{\varepsilon_{j-1}}y_{j-2}x^{\varepsilon_{j-2}}\dots
	y_kx_k^{\varepsilon_k}\\
&= Hy_nx_n^{\varepsilon_n}\dots
{\bf y_{j-1}x_jy_jy_{j-1}x_j}
	x_{j-1}^{\varepsilon_{j-1}}y_{j-2}x^{\varepsilon_{j-2}}\dots
	y_kx_k^{\varepsilon_k}\\
&= Hy_nx_n^{\varepsilon_n}\dots
{\bf y_jx_jy_{j-1}}
	x_{j-1}^{\varepsilon_{j-1}}y_{j-2}x^{\varepsilon_{j-2}}\dots
	y_kx_k^{\varepsilon_k}\\
&= Hy_nx_n^{\varepsilon_n}\dots
	y_jx_jy_{j-1}y_j
	x_{j-1}^{\varepsilon_{j-1}}y_{j-2}x^{\varepsilon_{j-2}}\dots
	y_kx_k^{\varepsilon_k}.
\end{align*}
Now we consider the action of $x_{k-1}$ on $Hy_nx_n^{\varepsilon_n}\dots
	y_kx_k^{\varepsilon_k}$. We prove by induction on $k$ starting with
	$k=n$ that 
	$$
	Hy_nx_n^{\varepsilon_n}\dots
	y_kx_k^{\varepsilon_k}\cdot x_{k-1} = Hy_nx_n^{\varepsilon_n}\dots
	y_kx_k^{\varepsilon_k}.
	$$
	For $k=n$ we have
	\begin{align*}
		Hy_n\cdot x_{n-1} &= Hx_{n-1} y_n = Hy_n\\
		Hy_nx_n\cdot x_{n-1}&=  H x_{n-1} x_ny_nx_{n-1}x_n =
		Hy_nx_{n-1}x_n = Hy_nx_n.
	\end{align*}
	Suppose $k\le n-1$ and $\varepsilon_{k}=0$. Then, since $x_{k-1}$
	commutes with all $y_s$, $s\ge k$ and all $x_s$, $s\ge k+1$ we get
$$
	Hy_nx_n^{\varepsilon_n}\dots
	y_kx_k^{\varepsilon_k}\cdot x_{k-1} = Hy_nx_n^{\varepsilon_n}\dots
	y_kx_k^{\varepsilon_k}.
	$$
	Suppose that $\varepsilon_k=1$. Then we get
	\begin{align*}
		Hy_nx_n^{\varepsilon_n}\dots
	y_kx_k\cdot x_{k-1} &= 	Hy_nx_n^{\varepsilon_n}\dots
	y_{k+1}x_{k+1}^{\varepsilon_{k+1}} {\bf x_{k-1}x_k y_k x_{k-1}x_ky_k} \\
&= 	Hy_nx_n^{\varepsilon_n}\dots
	y_{k+1}x_{k+1}^{\varepsilon_{k+1}} {\bf x_k y_k x_{k-1}x_ky_k}.
	\end{align*}
	Now by the induction hypothesis we have
	$$
		Hy_nx_n^{\varepsilon_n}\dots
		y_{k+1}x_{k+1}^{\varepsilon_{k+1}}\cdot x_{k} = Hy_nx_n^{\varepsilon_n}\dots
		y_{k+1}x_{k+1}^{\varepsilon_{k+1}}.
	$$
	Therefore
	\begin{align*}
			Hy_nx_n^{\varepsilon_n}\dots
	y_kx_k\cdot x_{k-1}  &= 	Hy_nx_n^{\varepsilon_n}\dots
	y_{k+1}x_{k+1}^{\varepsilon_{k+1}} {\bf  y_k x_{k-1}x_ky_k} \\
	&= 	Hy_nx_n^{\varepsilon_n}\dots
	y_{k+1}x_{k+1}^{\varepsilon_{k+1}} {\bf  y_k x_ky_k}\\
		&= 	Hy_nx_n^{\varepsilon_n}\dots
	y_{k+1}x_{k+1}^{\varepsilon_{k+1}}   y_k x_k\cdot y_k\\
	\end{align*}
	and we can use the computation for the action of $y_k$ on $Hy_nx_n^{\varepsilon_n}\dots
	y_kx_k^{\varepsilon_k}$.

	If $j=k$ then we get
$$
	Hy_nx_n^{\varepsilon_n}\dots
	y_kx_k^{\varepsilon_k}\cdot x_{k} = Hy_nx_n^{\varepsilon_n}\dots
	y_kx_k^{\varepsilon_k+1},
	$$
and this is a coset of allowed form. Suppose now that $k+1\le j\le n$. If
$\varepsilon_{j-1}=0$ then we get 
$$
	Hy_nx_n^{\varepsilon_n}\dots
	y_kx_k^{\varepsilon_k}\cdot x_{j} = Hy_nx_n^{\varepsilon_n}\dots
	y_jx_j^{\varepsilon_j}x_j y_{j-1}
	y_{j-2}x^{\varepsilon_{j-2}}\dots
	y_kx_k^{\varepsilon_k},$$
	which is also a coset of allowed form. Now assume that $\varepsilon_{j-1}=1$. Then
	we have
	\begin{align*}
		Hy_nx_n^{\varepsilon_n}\dots
	y_kx_k^{\varepsilon_k}\cdot x_{j} & = Hy_nx_n^{\varepsilon_n}\dots
	y_jx_j^{\varepsilon_j} y_{j-1}x_{j-1} x_j
	y_{j-2}x^{\varepsilon_{j-2}}\dots
	y_kx_k^{\varepsilon_k}\\  &=	
 Hy_nx_n^{\varepsilon_n}\dots
 y_jx_j^{\varepsilon_j} {\bf x_jx_{j-1}y_{j-1} x_jx_{j-1} y_{j-1}}
	y_{j-2}x^{\varepsilon_{j-2}}\dots
	y_kx_k^{\varepsilon_k} .
	\end{align*}
	From the previous computation we already know that
$$
	Hy_nx_n^{\varepsilon_n}\dots
	y_jx_j^{\varepsilon_{j+1}}\cdot x_{j-1} = Hy_nx_n^{\varepsilon_n}\dots
	y_jx_j^{\varepsilon_{j+1}}.
	$$
	Therefore by switching $x_{j}$ and $y_{j-1}$ we get
$$
Hy_nx_n^{\varepsilon_n}\dots
	y_kx_k^{\varepsilon_k}\cdot x_{j} = Hy_nx_n^{\varepsilon_n}\dots
 y_j{\bf x_j^{\varepsilon_j+2}  y_{j-1} x_{j-1} y_{j-1}}
	y_{j-2}x^{\varepsilon_{j-2}}\dots
	y_kx_k^{\varepsilon_k}.
$$
Using the relation $y_{j-1}x_{j-1}y_{j-1} = x_{j-1}y_{j-1}x_{j-1}$ we finally
get
$$
Hy_nx_n^{\varepsilon_n}\dots
	y_kx_k^{\varepsilon_k}\cdot x_{j} = Hy_nx_n^{\varepsilon_n}\dots
 y_j{ x_j^{\varepsilon_j}   y_{j-1} x_{j-1}}
	y_{j-2}x^{\varepsilon_{j-2}}\dots
	y_kx_k^{\varepsilon_k}.
$$
This completes the proof of the proposition.
\end{proof}
\begin{proposition}
	Suppose $\lambda_l=1$. Denote the group $\left\langle
	A_{2,n}(S_{\lambda|\mu'})
	\right\rangle$ by $H$. Then the cosets of $H$ in $\left\langle
	A_{2,n}(S_{\lambda|\mu})
	\right\rangle$ are given by the set 
	$$
	\left\{ Hy_ny_{n-1}\dots y_l y_{j_1}\dots y_{j_s}\ \middle|\  n-\mu_m+2\le l
	<j_1<j-2<\dots<j_s \le n\right\}.
	$$
	Note that the number of elements in this set is less then $2^{\mu_m-1}$.
\end{proposition}
\begin{proof}
		It is enough to show that the given set of cosets is closed under the
	action of generators $z\in S(\lambda,\mu)$.
	We denote the products $y_m\dots y_{k}$ by $w_{m,k}$. Let us fix a
	sequence $k+1\le j_1<\dots <j_s\le n$ and denote the product $y_{j_1}\dots
	y_{j_s}$ by $v$. We start by computing $Hw_{n,k}v\cdot x_i$. If
	$i<k$ then $vx_i=x_iv$ and $w_{n,k}x_i=x_iw_{n,k}$. Therefore
	$Hw_{n,k}vx_i= Hw_{n,k}v$. 

	Suppose $i\ge k$ and $i\not\in \left\{ j_1,\dots,j_s \right\}$.
	Then $vx_i=x_iv$. We claim that $Hw_{n,k}x_i=Hw_{n,k}y_{i+1}$. Then we
	will get that $y_{i+1}v$ is of allowed form since $y_i$ does not occur
	in $v$ and therefore $y_{i+1}$ can be brought in the suitable place.

	To show that $Hw_{n,k}x_i=Hw_{n,k}y_{i+1}$ we write $w_{n,k}$ as a
	product $w_{n,i}w_{i-1,k}$. Since $w_{i-1,k}x_i=x_iw_{i-1,k}$ and
	$y_{i+1}w_{i-1,k} = w_{i-1,k}y_{i+1}$ it is enough to show that
	$Hw_{n,i}x_i = Hw_{n,i} y_{i+1}$. 
	Note that $i<n$ as $\lambda_l=1$ and therefore
	$n=\lambda_1+\dots+\lambda_{l-1}\in \overline{s}(\lambda)$.
	Now we have
	\begin{align*}
		Hw_{n,i}x_i = Hw_{n,i+2}y_{i+1}y_i x_i = H w_{n,i+2}
		x_iy_iy_{i+1}x_iy_iy_{i+1}.
	\end{align*}
	Since $x_i$ and $y_i$ commute with $w_{n,i+2}$, and $x_i$ commutes with
	$y_{i+1}$ we get
	\begin{align*}
		Hw_{n,i}x_i = Hw_{n,i+2}y_{i+1}y_i x_i = H w_{n,i+2}
		y_{i+1}y_iy_{i+1} = Hw_{n,i}y_{i+1}
	\end{align*}
	as claimed.

	Now consider the case $i=j_t$ for some $1\le t\le s$. Let $i\ge m\ge k$
	be such that $m-1\not\in \left\{ j_1,\dots j_s \right\}$ but $\left\{
	m,m+1,\dots,j_t \right\}\subset \left\{ j_1,\dots,j_s \right\}$. Then we
	can write $v$ as a product $v'(w_{m,i})^{-1} v''$. We denote
	$(w_{m,i})^{-1}$ by $\overline{w}_{m,i}$. Note that $v'$ commutes with
	$\overline{w}_{m,i}$. Therefore we get
	$$
	Hw_{n,k} v\cdot x_i = Hw_{n,m-1} \overline{w}_{m,i} x_i w_{m-2,k} v'v''.
	$$
	We claim that $Hw_{n,m-1} \overline{w}_{m,i} x_i= Hw_{n,m-1}
	\overline{w}_{m,i}$. This will imply that $Hw_{n,k} v\cdot x_i=
	Hw_{n,k}v $. 

 First we consider the case $i=m$. In this case $n\ge m+1$ since $x_n\not\in
 S_{\lambda|\mu}$.  We have
 \begin{align*}
	 Hw_{n,m-1}\overline{w}_{m,i}\cdot x_i & = Hw_{n,m-1}y_m \cdot x_m =
	 Hw_{n,m+1}y_my_{m-1}y_mx_m \\
	 &= Hw_{n,m+1} y_m x_m y_m y_{m-1} x_m y_m y_{m-1} \\
	 &= Hw_{n,m+1} y_m x_m y_m x_m y_{m-1} y_m y_{m-1} \\
	 & = Hw_{n,m+1} x_m y_m y_{m-1} y_m y_{m-1} \\
	 & = Hx_m w_{n,m+1} y_{m-1} y_m y_{m-1} y_m \\
	 & = Hy_{m-1} w_{n,m+1} y_m y_{m-1} y_m = H w_{n,m-1} y_m.
 \end{align*}
 Now suppose that $i>m$. Then we have $\overline{w}_{m,i}=
 \overline{w}_{m,i-2}y_{i-1}y_i$, where $\overline{w}_{m,i-2}=e$ if $i=m+1$. 
 Therefore
 \begin{align*}
	 Hw_{n,m-1}\overline{w}_{m,i}\cdot x_i &  =
	 Hw_{n,m-1}\overline{w}_{m,i-2}y_{i-1} y_i x_i \\
	 &= Hw_{n,m-1} \overline{w}_{m,i-2} x_i y_i y_{i-1} x_i y_i y_{i-1} \\
	 &= H  w_{n,i+1} y_i y_{i-1} w_{i-2,m} \overline{w}_{m,i-2} x_i y_i y_{i-1} x_i y_i
	 y_{i-1} \\
	 &= H w_{n,i+1}  y_i y_{i-1} x_i y_i w_{i-2,m} \overline{w}_{m,i-2} y_{i-1} x_i y_i
\\&=
H w_{n,i+1} x_i y_{i-1} y_i x_i y_{i-1} w_{i-2,m} \overline{w}_{m,i-2} y_{i-1} x_i y_i
\\ & = H x_i y_{i-1} w_{n,i+1} y_i x_i w_{i-1,m} \overline{w}_{m,i-1} x_i y_i \\
&= Hw_{n,i} w_{i-1,m}\overline{w}_{m,i-1} x_i^2 y_i = Hw_{n,m}
\overline{w}_{m,i}.
 \end{align*}

 Now we consider the action of $y$'s on the given set of cosets. 
	Suppose $i\le k-2$ then $y_i$ commutes with both $w_{n,k}$ and
	$v$, and therefore $Hw_{n,k} v \cdot y_i = Hw_{n,k}v$.  Note that
	$y_{k-1}$ commutes with $v$ as neither $y_k$ nor $y_{k-2}$ occur in
	$v$. Therefore $Hw_{n,k}v\cdot y_{k-1} = Hw_{n,k}y_{k-1} v =
	Hw_{n,k-1}v$.  

	Suppose now that $k\le i\le n$. If $i+1$ is not an element of 
	$\{j_1,\dots,j_s\}$ then we can
	bring $y_i$ inside $v$ to the appropriate place and there is nothing
	to prove. If $i+1=j_{t}$ for some $t$ we consider two cases:
	$i=j_{t-1}$ and $i\not=j_{t-1}$. 

	If $j_{t-1}\not=i$ we can write $v$ as a product $v'y_{i+1}v''$, where
	$v'$ commutes with $y_{i+1}$ and $y_i$, and where $v''$ commutes with
	$y_i$. Therefore
	$$
	Hw_{n,k} v\cdot y_i = Hw_{n,k} y_{i+1} y_i v'v''.
	$$
	Now we have $w_{n,k} = w_{n,i-1} w_{i-2,k}$ and
	$y_{i+1}y_i$ commutes with $w_{i-2,k}$. We claim that
	$$
	Hw_{n,i-1} y_{i+1}y_i = Hw_{n,i-1} y_{i+1}.
	$$
	Then by backward substitution we will get that $Hw_{n,k} v \cdot y_i =
	Hw_{n,k} v$ in this case. 
We have
\begin{align*}
	Hw_{n,i-1}y_{i+1} y_i & = Hw_{n,i+2} y_{i+1} y_i y_{i-1}y_{i+1}y_i \\
	&= Hw_{n,i+2} y_{i-1} y_i y_{i+1} y_{i-1} y_i y_{i+1} y_{i-1}\\
	& = Hw_{n,i+2} y_{i+1} y_i y_{i-1} y_{i+1} = Hw_{n,i-1} y_{i+1}.
\end{align*}
Assume now that $j_{t-1}=i$. We can write $v$ as a product $v'y_iy_{i+1}v''$,
where $v''$ commutes with $y_i$. Then $Hw_{n,k} v y_i=
Hw_{n,k}v'y_iy_{i+1}y_iv''$. Now replace $y_iy_{i+1}y_i$ by
$y_{i+1}y_iy_{i+1}y_iy_{i+1}$. As $v'$ commutes with $y_{i+1}y_i$	
and since $Hw_{n,k} y_{i+1}y_i = Hw_{n,k}y_{i+1}$ we get
\begin{align*}
	Hw_{n,k}v'y_{i+1}y_iy_{i+1}y_iy_{i+1} v'' &=
	Hw_{n,k}y_{i+1}y_iv'y_{i+1}y_{i} y_{i+1} v''\\
	& = Hw_{n,k}y_{i+1} v' y_{i+1}y_i y_{i+1} v''\\
	& = Hw_{n,k}v'y_iy_{i+1} v'' = Hw_{n,k} v.
\end{align*}

\end{proof}
\bibliography{generators}
\bibliographystyle{amsplain}

\end{document}